\documentclass[12pt]{amsart}
 \usepackage{amsmath,amssymb,amsfonts,amsthm,amscd,textcomp,times,booktabs,mathrsfs}
   \usepackage[dvips]{graphicx}
    \usepackage{braket}
  \usepackage{color,float}
  \usepackage{bm}
\usepackage[all]{xy}
\theoremstyle{definition}
\newtheorem{theorem}{Theorem}[section]
\newtheorem{lemma}[theorem]{Lemma}

\newtheorem{proposition}[theorem]{Proposition}
\newtheorem{definition}[theorem]{Definition}
\newtheorem{remark}[theorem]{Remark}

\numberwithin{equation}{section}
\numberwithin{theorem}{section}

\newcommand{\del}{\ensuremath{\mathrm{\partial}}}
\def\C{\mathbb C}
\def\F{\mathcal{F}}
\def\G{\mathbb{G}}
\def\O{\mathcal{O}}
\def\P{\mathbb{P}}
\def\R{{\mathbb R}}
\def\S{\mathscr{S}}
\def\T{\mathscr{T}}
\def\la{\langle}
\def\ra{\rangle}
\def\conv{{\rm{Conv}}}
\def\sec{\displaystyle \Sigma{\rm{sec}}}

\begin{document}

\title{Facets of secondary polytopes and Chow stability of toric varieties}
%%% \subtitle{Sub-Title of Your Article}%%% optional

\author{Naoto \textsc{Yotsutani}}% Author Name (\sc should NOT be used here)
\address{School of Mathematical Sciences at Fudan University,
Shanghai, 200433, P. R. China}
\email{naoto-yotsutani@fudan.edu.cn}

% Subject code(s)

% Key word(s)

\begin{abstract}
Chow stability is one notion of
Mumford's Geometric Invariant Theory for studying the moduli space of
polarized varieties. Kapranov, Sturmfels and Zelevinsky detected
that Chow stability of polarized toric varieties is determined by
its inherent {\it secondary polytope}, which is a polytope whose
vertices correspond to regular triangulations of the
associated polytope \cite{KSZ}. In this paper, we give a
purely convex-geometrical proof that the Chow form of a
projective toric variety is $H$-semistable if and only if it is
$H$-polystable with respect to the standard complex torus action $H$.
This \emph{essentially} means that Chow semistability is equivalent to Chow polystability
for any (not-necessaliry-smooth) projective toric varieties.
\end{abstract}

\maketitle

\section{Introduction}
Let $X^n\longrightarrow \P^N$ be an $n$-dimensional complex projective variety with $\deg X\geqslant 2$ 
embedded by very ample complete linear system. 
Chow stability is one notion of Geometric Invariant Theory (GIT) investigated
by many researchers. In the present paper, we study Chow poly(semi)stability of a projective
toric variety for the standard complex torus action. To state our result more
precisely, let us briefly recall the fundamental knowledge on toric varieties.
See \cite{CLS,GKZ,Oda} for more details. 
Let $A=\set{{\bf{a}}_0,\dots, {\bf{a}}_N} \subset \mathbb{Z}^n$ be a finite set of integer vectors. Let $Q$ denote the convex hull of $A$ in $\mathbb{R}^n$. A finite set $A$ is said to {\emph {satisfy $(*)$}}  if the following conditions hold:
\begin{enumerate}
 \item[i)] $A=Q\cap \mathbb{Z}^n=\set{{\bf a}_0, \dots, {\bf a}_N}$. 
 \item[ii)] $A$ affinely generates the lattice $\mathbb{Z}^n$ over $\mathbb{Z}$. 
 \end{enumerate}
Now we regard $A$ as a set of Laurent monomials in $n$ variables, i.e., of monomials of the form
\[
{\bf x}^{\bf a}=x_1^{a_1}\cdots x_n^{a_n},
\] 
where ${\bf a}=(a_1, \dots, a_n)\in A$ is the exponent vectors and $x_1, \dots, x_n$ are $n$-variables.
The closure of the $A$-monomial embedding of a complex
torus $(\mathbb{C}^{\times})^n$ to the projective space defines the $n$-dimensional projective
toric variety $X_A$. It is well-known that toric Fano varieties with the anticanonical polarization correspond to
reflexive polytopes. Recall that a fully
dimensional integral polytope $Q$ containing the origin in its interior is called
 \emph{reflexive} if vertices are primitive lattice points and whose
polar dual polytope is again an integral polytope.

Next we quick review some related results on Chow stability of polarized
varieties which will be the source of our argument. One of the reasons why
Chow stability is important in K\"ahler geometry is that Chow stability is closely
related to the existence problem of canonical Riemannian metrics on a certain compact K\"ahler
manifold. A breakthrough result has been achieved by Donaldson in
\cite{Donaldson01}. Let $(X,L)$ be a smooth polarized variety, that is, $X$ is an $n$-dimensional smooth complex variety and
$L$ is a very ample line bundle over $X$. Donaldson showed that the existence of a constant scalar curvature
K\"ahler (cscK) metric representing the first Chern class $c_1(L)$ implies
asymptotic Chow stability of a polarized variety $(X,L)$ whenever $X$ has
no holomorphic vector fields. This result has been extended by Mabuchi in the case
where the automorphism group is not discrete. In \cite{Mabuchi}, Mabuchi proved
that if $(X,L)$ admits a cscK metric in $c_1(L)$ then $(X,L)$ is asymptotically Chow
polystable whenever $(X,L)$ satisfies the hypothesis of the obstruction for
asymptotic Chow semistability. Eventually, Futaki has detected that Mabuchi's
hypothesis is equivalent to the vanishing of a collection of integral invariants
$\F_{{\rm{Td}}^1},\dots,\F_{{\rm{Td}}^n}$ defined in \cite{Futaki}, where ${\rm{Td}}^i$
denotes the $i$-th Todd polynomial. The reader should bear in mind that $\F_{{\rm{Td}}^1}$ equals the
\emph{classical} Futaki invariant up to a multiplicative constant, so that $\F_{{\rm{Td}}^1}$ is an obstruction for the existence of
cscK metrics in $c_1(L)$. Since these integral
invariants $\F_{{\rm{Td}}^i}$ are a generalization of the classical Futaki invariant, we call them
\emph{higher} Futaki invariants. Combining Mabuchi's result \cite{Mabuchi} and Futaki's statement \cite{Futaki},
we have the following.
\begin{theorem}[Mabuchi-Futaki \cite{Mabuchi,Futaki}]
Let $(X,L)$ be an $n$-dimensional smooth polarized variety. Assume
that the higher Futaki invariants $\F_{{\rm{Td}}^i}$ vanishes for
each $i=1,\dots,n$. Then if $(X,L)$ admits a cscK metric in $c_1(L)$ then
$(X,L)$ is asymptotically Chow polystable.
\end{theorem}
One of the best possible result on the canonical Riemannian metrics of smooth toric Fano varieties,
due to X. J. Wang and X. Zhu, is the following.
\begin{theorem}[Wang-Zhu \cite{WZ}]
Let $X$ be a smooth toric Fano variety. Then
$(X,\O_X(K_X^{-1}))$ admits a K\"ahler-Einstein metric in $c_1(\O_X(K_X^{-1}))$ if and
only if the classical Futaki invariant vanishes.
\end{theorem}
Note that all cscK metrics in $c_1(\O_X(K_X^{-1}))$ are K\"ahler-Einstein metrics on
smooth Fano varieties. Summing up these results, one can see that asymptotic
Chow semistability implies asymptotic Chow polystability for \textit{smooth} toric Fano varieties. 
Considering a direct combinatorial proof of this result,
we provide more general result. That is, for an equivalently embedded projective toric variety $X_A \subset \P^N$, 
Chow semistability is \textit{essentially} equivalent to Chow polystability. In the above, the reader should bear in mind
that we fixed a polarization and do \textit{not} need asymptotic (semi)stability in order to show our result. More precisely, we have the following.
\begin{theorem}\label{th:main}
Let $A=\set{{\bf a}_0, \dots, {\bf a}_N} \subset \mathbb{Z}^n$ be a finite set of integer vectors which satisfies $(*)$.
 Let $X_A \longrightarrow \P^N$ be the associated complex projective toric variety with $\deg X_A\geqslant 2$. Considering the algebraic torus action of $(\C^{\times})^{N+1}$ into $\P^N$, we define the subtorus of $(\C^{\times})^{N+1}$ by
\[
H=\set{(t_0,\dots,t_N)\in (\C^{\times})^{N+1}|\prod_{j=0}^Nt_j=1}.
\]
Then the Chow point of $X_A$ is $H$-semistable if and only if it is $H$-polystable.
\end{theorem} 
Remark that Theorem $\ref{th:main}$ does \textit{not} require $X_A$ to be either smooth or Fano variety.
Also we note that $H$-polystability implies $H$-semistability by its definition (see Definition $\ref{def:stability}$).
On the other hand, even if $X$ is Fano variety Theorem $\ref{th:main}$ does not seem to be true in non-toric case. In fact, a cubic surface $X\subset \P^3$
with a singular point of type $A_2$ gives an example of Fano variety which is Chow semistable but {\it {not}} Chow polystable 
(see Remark $\ref{rem:cubic}$
for more details).

The main idea of our proof is based on the following observation. Let $G$ be a reductive
algebraic group. Suppose $G$ acts linearly on a finite dimensional complex vector space $V$.
The well-known Hilbert-Mumford numerical criterion of GIT (Proposition $\ref{prop:numerical criterion}$)
gives a necessary and sufficient condition for a nonzero vector $v^{\ast}\in V$ being polystable (resp. semistable).
 In the special case when the reductive group $G$ is isomorphic to the algebraic
 torus, this criterion can be restate in terms of the corresponding weight polytope (Proposition $\ref{prop:numerical criterion2}$).
See \cite{Dolga,GKZ,Gabor} for more details. Roughly speaking, the condition
for $H$-semistability in Theorem $\ref{th:main}$ is equivalent to the fact that the corresponding weight polytope $\mathcal{N}_H(X)$
with respect to $H$-action containing the origin. On the other hand, the
condition for $H$-polystability is equivalent to the fact that $\mathcal{N}_H(X)$ containing the origin in its interior. In particular, the weight polytope of the Chow point (form) of $X\hookrightarrow \P^N$ with respect to $(\C^{\times})^{N+1}$-action is called the {\it Chow polytope}. In the toric case, we can describe Chow polytopes in purely combinatorial way. Namely, the Chow polytope of a toric variety $X_A$ coincides with the {\it secondary polytope} $\sec (A)$, which is a polytope whose vertices are corresponding to regular triangulations of $Q$ 
(see Theorem $\ref{th:ksz}$).
We will use this combinatorial approach via secondary polytopes in order to show our main theorem.

This paper is organized as follows. Section $\ref{sec:GIT}$ is a brief review on Geometric
Invariant Theory and Chow stability. In Section $\ref{sec:GKZ}$, we first define the
secondary polytope and discuss about its fundamental property due to the work of
Gelfand, Kapranov and Zelevinsky. The structure of secondary polytopes is well-discussed
in \cite{GKZ,Triangulation}. Section $\ref{sec:facets}$ collects a combinatorial description on secondary polytopes. 
We give the proof of the main theorem in Section $\ref{sec:main}$. \\

\noindent{\it Acknowledgements.} This work was partially
developed while the author visited at L' Institut Henri Poincare in the Fall of 2012, to attend the special trimester in Conformal and K\"ahler Geometry. The author would like to thank the organizers for their financial support and kind hospitality during his visit. The author thanks Xiaohua Zhu, Hajime Ono and Yalong Shi for helpful discussions. In particular, Ono carefully read this article and gave me many helpful advices. The example of cubic surfaces in $\P^3$ (Remark $\ref{rem:cubic}$) is kindly taught me by Yalong Shi.
I am also grateful to the referee for valuable comments. Last but not least, I wish to thank Zhi Lu for his warm encouragement. 
The author is partially supported by the China Postdoctoral Science Foundation Grant, No. 2011M501045 and CAS Fellow for Young International Scientist, 2011Y1JB05.

\section{Preliminaries}\label{sec:GIT}
\subsection{Weight polytope}
Let $G$ be a reductive algebraic group and $V$ be a finite
dimensional complex vector space. Suppose $G$ acts linearly on $V$.
Let us denote a point $v^{\ast}$ in $V$ which is a representative of
$v\in \P(V)$.
\begin{definition}\label{def:stability}\rm
Let $v^{\ast}$ be as above and let
$\mathcal{O}_G(v^{\ast})$ be the $G$-orbit in $V$.
\begin{itemize}
\item[(a)] $v^{\ast}$ is called \emph{$G$-semistable} if the Zariski
closure of $\O_G(v^{\ast})$ does not contain the origin: $0\notin
\overline{\O_G(v^{\ast})}$.
\item[(b)] $v^{\ast}$ is called \emph{$G$-polystable} if
$\mathcal{O}_G(v^{\ast})$ is closed orbit.
\end{itemize}
Analogously, $v\in \P(V)$ is said to be $G$-polystable (resp.
semistable) if any representative of $v$ is $G$-polystable (resp. semistable).
\end{definition}
\begin{remark}
The closure of $\O_G(v^{\ast})$ in the Euclidean topology coincides with the
Zariski closure $\overline{\O_G(v^{\ast})}$ (see, \cite{Mum76}, Theorem $2.33$).
\end{remark}

From Definition $\ref{def:stability}$, one can see that $G$-polystability
implies $G$-semistability as $G$-orbit itself never contain the
origin. The following Hilbert-Mumford criterion is well-known in
Geometric Invariant Theory.

\begin{proposition}[The Hilbert-Mumford criterion \cite{MFK94}]\label{prop:numerical criterion}
$v\in \P(V)$ is $G$-polystable (resp. semistable) if and only if $v$
is $H$-polystable (resp. semistable) for all maximal algebraic torus
$H\leq G$.
\end{proposition}
Now we assume that a reductive group $G$ is isomorphic to an algebraic torus.
Let $\chi(G)$ denote the character group of $G$. Then $\chi(G)$ consists of algebraic homomorphisms
$\chi: G\longrightarrow \C^{\times}$. If we fix an isomorphism $G\cong (\C^{\times})^{N+1}$, we may
express each $\chi$ as a Laurent monomial
\[
\chi(t_0,\dots, t_N)=t_0^{a_0}\cdots t_N^{a_N},\quad 
t_i \in \C^{\times}, \; a_i\in \mathbb{Z}.
\]
Thus, there is the identification between $\chi(G)$ and $\mathbb{Z}^{N+1}$:
$$\chi=(a_0,\dots,a_N)\in \mathbb{Z}^{N+1}.$$
Then it is well-known that $V$ decomposes under the action of $G$ into weight spaces
\[
V=\bigoplus_{\chi \in \chi(G)}V_{\chi}, \qquad
V_{\chi}:=\set{v^{\ast}\in V | t\cdot v^{\ast} = \chi(t)\cdot
v^{\ast},\; t\in G}.
\]
\begin{definition}[Weight polytope]\rm
Let $v^{\ast}\in V\setminus \set{{\bf0}}$ be a nonzero vector in $V$ with
$$v^{\ast}=\sum_{\chi\in \chi(G)}v_{\chi},\qquad v_{\chi}\in V_{\chi}.$$
The \emph{weight polytope} of $v^{\ast}$ (with respect to $G$-action) is the
integral convex polytope in $\chi(G)\otimes \R\cong \R^{N+1}$
defined by
\[
\mathcal{N}_G(v^{\ast}):=\conv\set{\chi\in \chi(G)|v_{\chi}\neq
0}\subset \R^{N+1},
\]
where $\conv\set{A}$ denotes the convex hull of a finite set of
points $A$.
\end{definition}
In the case where $G$ is an algebraic torus, the Hilbert-Mumford criterion (Proposition $\ref{prop:numerical criterion}$) can be
restated as the following proposition.
\begin{proposition}[{\emph{The numerical criterion}} : \cite{Dolga} Theorem 9.2 ,
\;\cite{Gabor} Theorem 1.5.1]\label{prop:numerical criterion2}
Suppose $G$ is isomorphic to an algebraic torus which acts a complex vector space $V$ linearly. Let $v^{\ast}$ be a nonzero vector
in $V$.
Then
\begin{itemize}
\item[(i)] $v^{\ast}$ is $G$-semistable if and only if
$\mathcal{N}_G(v^{\ast})$ contains the origin. \ \\
\item[(ii)] $v^{\ast}$ is $G$-polystable if and only if
$\mathcal{N}_G(v^{\ast})$ contains the origin in its interior.
\end{itemize}
\end{proposition}

\subsection{Chow form}
Now we recall the definition of the Chow form of irreducible complex
projective varieties. See \cite{GKZ} for more details.

Let $X\longrightarrow \P^N$ be an $n$-dimensional irreducible
complex projective variety of degree $d\geqslant 2$. Recall that the
Grassmann variety $\mathbb{G}(k,\P^N)$ parameterizes $k$-dimensional
projective linear subspaces of $\P^N$.
\begin{definition}[Associated hypersurface]\rm
The \emph{associated hypersurface} of $X\longrightarrow \P^N$ is the
subvariety in $\mathbb{G}(N-n-1, \P^N)$ which is given by
\[
Z_X:=\set{L\in \G(N-n-1,\P^N)| L\cap X\neq \emptyset}.
\]
\end{definition}
The fundamental properties of $Z_X$ can be summarized as follows (see
\cite{GKZ}, p.$99$):
\begin{itemize}
\item[(1)] $Z_X$ is irreducible,
\item[(2)] $\mathrm{Codim\:} Z_X=1$ (that is, $Z_X$ is a divisor in
$\G(N-n-1,\P^N)$),
\item[(3)] $\deg Z_X=d$ in the Pl\"ucker coordinates, and
\item[(4)] $Z_X$ is given by the vanishing of a section $R_X^{\ast}\in
H^0(\G(N-n-1,\P^N),\O(d))$.
\end{itemize}
We call $R_X^{\ast}$ the \emph{Chow form} of $X$. Note that
$R_X^{\ast}$ can be determined up to a multiplicative constant. Setting
$V:=H^0(\G(N-n-1,\P^N),\O(d))$ and $R_X \in \P(V)$ which is the
projectivization of $R_X^{\ast}$, we call $R_X$ the \emph{Chow point} of
$X$. Since we have the natural action of $G={\rm{SL}}(N+1,\C)$ into
$\P(V)$, we can define ${\rm{SL}}(N+1)$-polystability
(resp. semistability) of $R_X$.
\begin{definition}[Chow stability]\rm
Let $X\longrightarrow \P^N$ be an irreducible, $n$-dimensional
complex projective variety. Then $X$ is said to be \emph{Chow
polystable (resp. semistable)} if the Chow point $R_X$ of $X$ is
${\rm{SL}}(N+1,\C)$-polystable (resp. semistable).
\end{definition}

\section{Secondary polytopes and Regular triangulations}\label{sec:GKZ}
\subsection{A construction of secondary polytopes}
In this section we recall the definition of secondary polytope
and its fundamental property. For more details, see
\cite{GKZ,Triangulation}.
 
Let $A=\set{{\bf{a}}_0,\dots,{\bf{a}}_N}$ be a finite subset in $\mathbb{Z}^n$ which satisfies $(*)$. Let $Q$ be the convex hull of $A$ in $\R^n$ as usual. To begin, we construct the \emph{regular triangulation} of $(Q,A)$ as follows: \\
\begin{description}
\item[Step1. (Lifting)] Pick a \emph{height function} $\omega: A \longrightarrow \R$ which can be thought of as a vector
$\omega=(\omega_0,\dots,\omega_N)\in \R^{N+1}$ with $\omega({\bf{a}}_i)=\omega_i$. Using the coordinate of $\omega$ as `heights', we consider the \emph{lifted finite set} in $\R^{n+1}$, defined by
\[
A^{\omega}:=\set{{\hat {\bf{a}}}_0, \dots ,\hat{\bf{a}}_N} \subset \R^{n+1}, \quad \hat{\bf{a}}_i=
\begin{pmatrix}
{\bf{a}}_i \\
\omega_i
\end{pmatrix}.
\]
\item[Step2. (Lower Face)] 
Let $Q^{\omega}$ be the convex hull of $A^{\omega}$ in $\R^{n+1}$. A face $F$ of $Q^{\omega}$ is said to be a \emph{lower face} if it satisfies
\[
{\bf x}-c{\bf e}_{n+1}\notin Q^{\omega} \quad \text{for each}\quad {\bf x}\in F \quad \text{and}\quad c>0.
\]
Here ${\bf e}_{n+1}=(0,\dots,0,1) \in \R^{n+1}$.\vspace{0.3cm}
\item[Step3. (Projection)] Let $p$ denote the canonical projection
\[
p: \R^{n+1} \longrightarrow \R^n \qquad (x_1,\dots, x_{n+1})
\longmapsto (x_1,\dots, x_n).
\]
Then, if all lower faces of $Q^{\omega}$ are simplices, their projections
\[
\set{p(F)| F \text{ is a lower face of }
Q^{\;\omega}}
\]
form a triangulation of $(Q,A)$.
\end{description}

\begin{definition}[Regular Triangulation]\label{def:RegTri}\rm
Let $A$ and $Q$ be as above. A triangulation of $(Q,A)$ is called \emph{regular} if it can be obtained by projecting all the
lower faces of a lifted finite set $A^{\omega}$ in $\R^{n+1}$ for some $\omega \in \R^{N+1}$.
\end{definition}
Let $A$ and $Q$ be as above and let $T$ be a triangulation of $(Q,A)$. Let $J=\set{0, \dots, N}$ be
the index set of labels. Fix a point ${\bf{a}}_j \in A$. Let ${\rm{Vol}} (\cdot)$ denote a translation invariant volume form on $\R^n$ with the
normalization ${\rm{Vol}} (\Delta_n)=1/{n!}$ for the standard $n$-dimensional simplex
 $\Delta_n=\conv \set{{\bf{e}}_i | 1\leqslant i \leqslant n}$. 
For any simplex $C$ of $T$, we denote the set of vertices of $C$ by $\mathcal{V}(C)$. Then we consider the function 
$\phi_{A,T}: A\longrightarrow \R$ defined by
\[
\phi_{A,T}({\bf{a}}_j)=\sum_{C: {\bf{a}}_j \in \mathcal{V}(C)}
n!{\rm{Vol}}(C)
\]
where the summation is over all maximal simplices of $T$ for which ${\bf{a}}_j$ is a vertex. Especially, $\phi_{A,T}({\bf{a}}_j)=0$ for $j\in J$
if and only if ${\bf{a}}_j\in A$ is not a vertex of any simplex of $T$. Then the \emph{Gelfand-Kapranov-Zelevinsky (GKZ) vector} of $T$ is given by
\[
\phi_A(T)=\sum_{j\in J}\phi_{A,T}({\bf{a}}_j){\bf{e}}_j \in \R^{N+1}
\]
where ${\bf{e}}_j$ for $j\in J$ is the standard basis of $\R^{N+1}$. 
\begin{definition}[Secondary polytope]\rm
The \emph{secondary polytope} $\sec(A)$ is the polytope in $\R^{N+1}$ defined
by
\[
\sec(A)=\conv\set{\phi_A(T)| T \text{\;is a triangulation of\;}
 (Q,A)}.
\]
\end{definition}
The following properties of secondary polytopes are well-known.
\begin{theorem}[\cite{GKZ} p.$221$, Theorem $1.7$]\label{th:gkz}
For a finite subset $A=\set{{\bf{a}}_0, \dots, {\bf{a}}_N}$ in $\mathbb{Z}^n$ which satisfies $(*)$, we have
\begin{itemize}
\item[(i)] $\dim \sec(A)=N-n$.
\item[(ii)] There is a one to one correspondence between the regular
triangulations of $(Q,A)$ and vertices of $\sec(A)$. In particular, the
GKZ-vector $\phi_A(T)$ for a triangulation $T$ of $(Q,A)$ will be
a vertex of $\sec(A)$ if and only if $T$ is regular.
\end{itemize}
\end{theorem}
In order to see the relationship between secondary polytopes and Chow polytopes of toric varieties, we first quick review on
the construction of toric varieties. See \cite{GKZ}, Chapter $5$ for more details. Recall that a toric variety is a complex irreducible algebraic variety with a complex torus action having an open dense orbit. As usual, let $A=\set{{\bf a}_0,\dots ,{\bf a}_N}$ be a finite set of integer vectors in $\mathbb{Z}^n$ which satisfies $(*)$. Setting
\[
X_A^0=\set{[{\bf{x}}^{{\bf{a}}_0}:\dotsc :{\bf{x}}^{{\bf{a}}_N}]\in
\P^N| {\bf{x}}=(x_1,\dots,x_n)\in (\C^{\times})^n},
\]
we define the variety $X_A\subset \P^N$ to be the closure of $X_A^0$
in $\P^N$. Then $X_A$ is an $n$-dimensional equivariantly embedded subvariety in $\P^N$. Then we require the following result.
\begin{theorem}[Kapranov-Sturmfels-Zelevinsky
\cite{KSZ}]\label{th:ksz} 
Let $A\subset \mathbb{Z}^n$ be a finite set which satisfies $(*)$. Let $X_A \subset \P^N$ be the associated toric variety.
Let $R_{X_A}$ be the Chow point of $X_A$. Then the weight polytope
$\mathcal{N}_{(\C^{\times})^{N+1}}(R_{X_A})$ of $R_{X_A}$ with respect
to the algebraic torus action $(\C^{\times})^{N+1}$ (i.e., the \emph{Chow polytope} of $X_A$) coincides with the secondary
polytope $\sec(A)$.
\end{theorem}
Next we define a piecewise-linear function $g_{\omega,T}: Q\longrightarrow \R$ as follows.
Let $T$ be a triangulation of $(Q,A)$ and let $\omega \in \R^{N+1}$ be a
height function. The \emph{characteristic section} of $T$ with
respect to $\omega$ is a piecewise-linear function which is defined
by
\[
g_{\omega,T}: Q\longrightarrow \R \qquad {\bf{a}}_i \longmapsto
g_{\omega,T}({\bf{a}}_i)=\omega_i
\]
and extended affinely on $C$ for each maximal simplex $C$ of $T$. Remark that in the definition of the characteristic section, 
we do not require $\omega$ to be the height function that induces the triangulation $T$. 

\begin{proposition}[\cite{GKZ} p.$221$, Lemma $1.8$]\label{prop:gkz1} 
Let $\omega \in \R^{N+1}$ be a height
function and let $T$ be any triangulation of $(Q,A)$. For the
characteristic section $g_{\omega,T}$ of $T$ with respect to $\omega$ and the GKZ-vector
$\phi_A(T)$, we have
\begin{align*}
\la \omega, \phi_A(T) \ra  =& (n+1)!\int_Q g_{\omega,T}({\bf{x}})dv.
\\
\end{align*}
\end{proposition}

We finish this subsection with the following lemma.

\begin{lemma}\label{lem:triangulation}
Let $\omega, T, g_{\omega,T}$ and $\phi_A(T)$ be as in Proposition $\ref{prop:gkz1}$.  For each maximal simplex
$C$ of $T$, we have
\begin{equation}\label{eq:IntSimp}
\int_C g_{\omega,T}({\bf{x}})dv=\frac{{\rm{Vol}}(C)}{n+1}\sum_{j\in
C}\omega_j.
\end{equation}
Moreover,
\begin{equation}\label{eq:GKZVol}
\la \omega, \phi_A(T) \ra  = n!\sum_{C\in T}{\rm{Vol}}(C)\sum_{j\in C}\omega_j
\end{equation}
where the first summation runs over all maximal simplices of $T$.
\end{lemma}
\begin{proof}
\eqref{eq:GKZVol} follows from \eqref{eq:IntSimp} and Proposition $\ref{prop:gkz1}$. Hence it suffices to show \eqref{eq:IntSimp}.

From the definition of $g_{\omega,T}$, we have $g_{\omega,T}({\bf{a}}_j)=\omega_j$.
Note that the integral of a linear function on a domain is equal to
the multiplication of the volume of a domain with the value of a
linear function at the barycenter. In our case, this implies
\begin{equation}\label{eq:piecewise lin fun2}
\int_{C}g_{\omega,T}({\bf{x}})dv={\rm{Vol}}(C)
g_{\omega,T}({\bf{b}}_C),
\end{equation}
where ${\bf{b}}_C$ is the barycenter of a simplex $C$. Now we use the
fact that the barycenter of a simplex is given by the average of its
vertices:
\begin{equation}\label{eq:barycenter}
{\bf{b}}_C:=\int_{C}{\bf{x}}dv=\frac{1}{n+1}\sum_{{\bf{a}}\in
\mathcal V(C)}{\bf{a}}.
\end{equation}
Note that we have the linearity of $g_{\omega,T}$ with respect to the barycenter (see \cite{Triangulation}, p.$219$).
Substituting \eqref{eq:barycenter} in \eqref{eq:piecewise lin fun2}, we have \eqref{eq:IntSimp}.
The assertion is verified.
\end{proof}

\subsection{Facets of the secondary polytope}\label{sec:facets}
We describe the faces of secondary polytopes. The facets
(i.e., codimension $1$ faces) of $\sec(A)$ correspond to maximal regular subdivisions of $(Q,A)$.
These are called \emph{coarse subdivisions} (see \cite{GKZ}, Chapter $7$ section $2$ B). 

Let $A=\set{{\bf{a}}_0,\dots, {\bf{a}}_N}$ be a finite subset in $\mathbb{Z}^n$ which satisfies $(*)$ and let $Q$ be the convex hull of $A$ in $\R^n$. Recall that a subdivision of $(Q,A)$ is called \emph{regular} if it can be obtained by projecting all the lower faces of a lifted finite set
$A^{\omega}$ for some $\omega$ (Definition $\ref{def:RegTri}$). Let $\S(A,\omega)$ denote the regular subdivision of $(Q,A)$ produced by $\omega$. We will find the defining equation of the facet of $\sec(A)$ corresponding to a certain coarse subdivision $\S(A,\omega)$.
To begin, we shall define a \emph{refinement} of a polyhedral subdivision.
\begin{definition}[Refinement]\rm
Let $S$ and $S'$ be two subdivisions of $(Q,A)$. Then $S$ is said to
be a \emph{refinement} of $S'$ if for any $C\in S$, there is a $C'\in S'$ with $C\subseteq C'$.
We denote it by $S \preceq S'$.
\end{definition}
The following theorem due to Gelfand, Kapranov and Zelevinsky gives a combinatorial description of
 the faces of secondary polytopes. (cf. Theorem $\ref{th:gkz}$).
\begin{theorem}[\cite{GKZ} p.228, Theorem 2.4]\label{th:gkz2} Let $S$ be any regular subdivision of $(Q,A)$.
Let $F(S)$ denote the convex hull in $\R^{N+1}$ of the GKZ-vectors
for all triangulations $T$ which is obtained by refining $S$:
\[
F(S)=\conv \set{\phi_A(T)|T \text{\;is a triangulation refining\;}
S}.
\]
Then two faces of $\sec(A)$ satisfy $F(S)\subset F(S')$ if and only
if $S\preceq S'$.
\end{theorem}
From Theorem $\ref{th:gkz2}$, the facets of the secondary polytope
$\sec(A)$ correspond to regular subdivisions of $(Q,A)$ which only
refine the trivial subdivision and no other. We call these subdivisions
the \emph{coarse subdivisions}. Note that the trivial subdivision
always exists and is given by the zero height function $\omega =(0, \dots ,0)$.
The following Lemma gives the explicit defining equation of the facet
of $\sec(A)$ corresponding to a coarse subdivision.
\begin{lemma}[\cite{Triangulation} Excercise 5.11]\label{lem:gkz2}
Let $(Q,A)$ be as above. Let $\omega \in \R^{N+1}$ be a height function which produces the coarse subdivision
$\S(A,\omega)$ of $(Q,A)$. The defining linear equation of the facet of $\sec(A)$
corresponding to $\S(A,\omega)$ is 
\[
\sum_{j\in J}{\omega_j}\varphi _j =n! \sum_{C\in T}{\rm{Vol}} (C)\sum_{j\in C}
\omega_j \qquad \text{for}\quad \varphi=(\varphi_0,\dots,\varphi_N)\in \R^{N+1},
\]
where $T$ is a certain triangulation which is obtained by
refining $\S(A,\omega)$ (i.e., $T \preceq \S(A,\omega)$).
\end{lemma}

\section{Proof of the main theorem}\label{sec:main}
Now we ready to prove Theorem $\ref{th:main}$.
\begin{proof}
Let $A=\set{{\bf{a}}_0,\dots,{\bf{a}}_N}\subset \mathbb{Z}^n$ be a finite subset which satisfies $(*)$ and let $J=\set{0,\dots, N}$ be the index set of labels. Let
$X_A\longrightarrow \P^N$ be the associated projective toric variety of degree $d=\deg X_A \geqslant 2$. 
We denote the Chow point of $X_A$ by $R_{X_A}$. Considering the complex torus
$(\C^{\times})^{N+1}$, we define the subtorus
of $(\C^{\times})^{N+1}$ by
\[
H=\set{(t_0,\dots,t_N)\in (\C^{\times})^{N+1}|\prod^N_{j=0}t_j=1}\cong (\C^{\times})^N.
\]
Suppose that $R_{X_A}$ is $H$-semistable but not $H$-polystable.
Setting $G=(\C^{\times})^{N+1}$, we consider the projection
\begin{align}\label{map:projection}
\begin{split}
\pi_H: \chi(G)\otimes \R \cong \R^{N+1} &\longrightarrow  \chi(H)\otimes \R \cong \R^N,\\
(\varphi_0, \dots, \varphi_N) &\longmapsto (\varphi_0-\varphi_N, \dots, \varphi_{N-1}-\varphi_N).
\end{split}
\end{align}
Then by Theorem $\ref{th:ksz}$ we observe that
\begin{equation*}\label{projection}
\pi_H(\sec(A))=\mathcal{N}_H(R_{X_A}) \qquad \text{and}\qquad
\pi_H^{-1}(\del \mathcal{N}_H(R_{X_A})) \subset \del \sec(A),
\end{equation*}
where $\del P$ denotes the boundary of an integral polytope $P$. 
Thus, the numerical criterion (Proposition
$\ref{prop:numerical criterion2}$) implies that there is an element
$\varphi=(\varphi_0,\dots,\varphi_N)$ in $\del \sec(A)$ satisfying
$\pi_H(\varphi)={\bf{0}}\in \partial \mathcal{N}_H(R_{X_A})$. In particular, there exists $t\in \R$ such that
\begin{equation}\label{eq:BoundSec}
(\underbrace{t,\cdots ,t}_{N+1})\in \del \sec(A)
\end{equation}
from \eqref{map:projection}. Meanwhile, we have the equality
\begin{equation}\label{eq:Ono}
(N+1)t=(n+1)!\mathrm{Vol}(Q)
\end{equation}
by ($18$) in \cite{Ono11}. This implies that $t\neq 0$ as $\mathrm{Vol}(Q)>0$ in \eqref{eq:Ono}. Hence we may assume that
there exists $t\in \R^{\times}$ satisfying \eqref{eq:BoundSec}. 

Now we take the facet $F$ of
$\sec(A)$ which contains the point $(t,\cdots,t)$ in \eqref{eq:BoundSec}. As discussed in Section
$\ref{sec:GKZ}$, there is a height function $\omega\in \R^{N+1}$ which
produces the coarse subdivision $\S(A,\omega)$ corresponding to
this facet $F$. Fix $\omega \in \R^{N+1}$. Then Lemma $\ref{lem:gkz2}$ implies that 
\[
t\sum_{j\in J}\omega_j =n! \sum_{C\in \T}{\rm{Vol}}(C)\sum_{j\in C}\omega_j
\]
for a certain triangulation $\T\preceq \S(A,\omega)$ which is given by a refinement of $\S(A,\omega)$. Also, Lemma
 $\ref{lem:triangulation}$ gives
 \[
\la \omega, \phi_A(T)\ra =n!\sum_{C\in T}{\rm{Vol}}(C)\sum_{j\in C}\omega_j
 \]
for any triangulation $T$ of $(Q,A)$. Taking $T=\T\preceq \S(A,\omega)$, we have
\begin{equation}\label{eq:relation1}
t\sum_{j\in J}\omega_j =\la \omega, \phi_A(T)\ra.
\end{equation}

On the other hand, we may assume that there is a subset $I \subset J$ such that
$$
\begin{cases}
\omega_i=1 & \text{for}\; \; i\in I, \\
\omega_j=0 & \text{for}\; \; j\in J\setminus I,
\end{cases}
$$
from the definition of the coarse subdivisions.

{\bf{Case I. The simplest case:}}
 Assume that there is only one $i\in I$ satisfying $\omega_{i}=1$ and there are no other (i.e., $I=\set{i}$). 
Then we have the following two possibilities:
(a) ${\bf{a}}_{i} \notin \mathcal{V}(Q)$ and (b) ${\bf{a}}_{i} \in \mathcal{V}(Q)$.

In the case of (a), we observe that
\[
\langle \omega, \phi_A(T) \rangle=0
\]
for any $T\preceq \S(A,\omega)$. Remark that ${\bf{a}}_{i}$ is never a vertex of any simplices of $T$ because ${\bf{a}}_{i}$ is lifting by the height function
$\omega$. Therefore \eqref{eq:relation1} implies $t=0$. This contradicts $t\in \R^{\times}$.

In case (b), ${\bf{a}}_{i}$ must be contained in a standard simplex $C$ of $T$ with ${\bf{a}}_{i}\in \mathcal{V}(C)$, where $\mathrm{Vol}(C)=1/n!$. Thus,
\[
\langle \omega, \phi_A(T) \rangle =\omega_i\cdot \phi_{A,T}({\bf{a}}_i)=1.
\]
Then \eqref{eq:relation1} implies $t=1$. 
Substituting this in \eqref{eq:Ono}, we have
\begin{equation}\label{eq:relation2}
N+1=(n+1)!\mathrm{Vol}(Q).
\end{equation} 
Therefore, Lemma $\ref{lem:polytope}$ (see the Appendix) implies that 
$Q$ is a standard $n$-dimensional simplex $\Delta_n=\linebreak{\rm{Conv}}\set{{\mathbf{e}}_i| 1\leqslant i \leqslant n}$.
Then the associated toric variety is
$(\P^n,\mathcal{O}_{\P^n}(1))$ which has degree $1$. This contradicts $\deg X_A\geqslant 2$.

Hence the assertion is verified in the simplest case.

 {\bf{Case II. The general case:}}
Now let us consider the general case. For the simplicity, we may assume that
\[
\omega=(0,\dots,0,\omega_i,0,\dots,0,\omega_{i'},0,\dots,0)
\]
for $\omega_i=\omega_{i'}=1$. Other cases are similar and our proof is readily generalized to such cases with minor modifications.
Then we have the following three possibilities:

(a) In the case where both ${\bf{a}}_i$ and ${\bf{a}}_{i'}$ are not a vertex of $Q$, we conclude that
\[
\langle \omega, \phi_A(T) \rangle =0
\]
by the computation in Case I-(a). Again, this yields $t=0$, a contradiction.

(b) In the case where ${\bf{a}}_i\in \mathcal{V}(Q)$ but ${\bf{a}}_{i'}\notin \mathcal{V}(Q)$, we have
\begin{align*}
\langle \omega, \phi_A(T) \rangle &=\omega_i\cdot \phi_{A,T}({\bf{a}}_i)+\omega_{i'}\cdot \phi_{A,T}({\bf{a}}_{i'}) \\
&=1,
\end{align*}
by the same argument in Case I-(a),(b). Therefore $t=\frac{1}{2}$ by \eqref{eq:relation1}. Substituting this in \eqref{eq:Ono}, we have
\[
(N+1)=2(n+1)!\mathrm{Vol}(Q).
\]
This contradicts Lemma $\ref{lem:polytope}$.

(c) In the case where both ${\bf{a}}_i$ and ${\bf{a}}_{i'}$ are vertices of $Q$, we have 
\[
\langle \omega, \phi_A(T) \rangle =2.
\]
Thus \eqref{eq:relation1} implies $t=1$. Now we repeat the argument in the last part of Case I-(b). 

The proof is complete.
\end{proof}
\begin{remark}\label{rem:cubic}
It is an interesting problem to generalize Theorem $\ref{th:main}$ to the case of \textit{non-toric}. However, the following example indicates that
there seems to be no such a generalization even to the case of \emph{Fano} varieties
 (see \cite{Mukai}, $7.2$ (b) for more details). 

Let $X$ be a cubic surface in $\P^3$ and let $R_X$ denote the Chow point of $X$. Remark that $R_X$ is given by the defining equation of $X$ since $X$ is
a hypersurface. We recall the following results on Mumford's Geometric Invariant Theory:
\begin{itemize}
\item $X$ is Chow stable if and only if it has finitely many singular points of type $A_1$ and no worse singularities;
\item $X$ is Chow semistable if and only if it has at most finitely many singular points of type $A_1$ or type $A_2$.
\end{itemize} 
Let $X_0\subset \P^3$ be a special one which is given by
\[
X_0:=\set{[x:y:z:w]\in \P^3|y^3-xzw=0}\subset \P^3.
\]
Then $X_0$ has exactly three singular points
\[
p_1=[1:0:0:0],\quad p_2=[0:0:1:0],\quad p_3=[0:0:0:1]
\]
which are all of type $A_2$. Thus, $X_0$ is Chow semistable and not Chow stable. 
Moreover, it is well-known that $SL(4,\C)\cdot R_{X_0}$ is a closed orbit (\cite{Mukai}, Proposition $7.23$).
Hence we conclude that $X_0$ is Chow polystable. Let us consider any other cubic surface $X$ with a singular point
of type $A_2$ such that
\[
SL(4,\C)\cdot R_X \cap SL(4,\C)\cdot R_{X_0}=\emptyset.
\] 
Obviously, $X$ is Chow semistable. Then it follows that the closure of $SL(4,\C)\cdot R_X$ contains $R_{X_0}$. This implies that
$SL(4,\C)\cdot R_X$ is not closed orbit. Therefore, $X$ is Chow semistable but not Chow polystable.
\end{remark}

\noindent{{\bfseries Appendix.} }\vspace{0.4cm}

In this appendix we shall show Lemma $\ref{lem:polytope}$ which is used in the proof of our theorem. To begin, we recall some properties of
the \emph{Ehrhart {\rm{h}}-vector} of an integral polytope. See \cite{Mustata}, \cite{BH}, Chapter $6$, for more details.

Let $Q\subset \R^n$ be an $n$-dimensional integral polytope. Let $E_Q(t)$ denote the \textit{Ehrhart polynomial} of $Q$, which is a polynomial
of degree $n$ satisfying
\[
E_Q(\ell)=\mathrm{Card}(\ell Q \cap \mathbb{Z}^n)
\]
for each positive integer $\ell$. Then we define its \textit{Ehrhart series} by
\[
\mathrm{Ehr}_Q(t):=1+\sum_{\ell \geqslant 1} E_Q(\ell)t^{\ell}.
\]
It is well-known that $\mathrm{Ehr}_Q(t)$ can be written as the power series expansion at $t=0$ of a rational function
\[
\frac{h_nt^n+h_{n-1}t^{n-1}+\cdots +h_0}{(1-t)^{n+1}}
\]
with some integers $h_0,\dots, h_n$. We call $(h_0,\dots, h_n)$ the \textit{Ehrhart {\rm{h}}-vector} of $Q$.
Then the Ehrhart $h$-vector satisfies the following properties.
\begin{proposition}[Ehrhart-Stanley]\label{prop:EH}
Let $Q$ be an $n$-dimensional integral polytope in $\R^n$.
\begin{enumerate}
\item $h_0=1,\quad h_1=\mathrm{Card}(Q\cap \mathbb{Z}^n)-n-1$.
\item $\displaystyle n!\mathrm{Vol}(Q)=\sum_{\ell=0}^nh_\ell$.\\
\item $h_{\ell} \in \mathbb{Z}_{\geqslant 0} \quad \text{for} \quad 0 \leqslant \ell \leqslant n$.
\end{enumerate}
\end{proposition}
\begin{lemma}\label{lem:polytope}
Let $Q$ be an $n$-dimensional integral polytope in $\R^n$. Then we have
\begin{equation*}\label{ineq:polytope}
\mathrm{Card}(Q\cap \mathbb{Z}^n) \leqslant (n+1)! \mathrm{Vol}(Q)
\end{equation*}
and equality holds if and only if $Q$ is the standard $n$-simplex $\Delta_n$.
\end{lemma}
\begin{proof}
Let $(h_0,\dots,h_n)$ be the Ehrhart $h$-vector of $Q$.
Combining $(1)$ and $(2)$ in Proposition $\ref{prop:EH}$, we have
\begin{align*}
(n+1)!\mathrm{Vol}(Q)&=(n+1)\cdot n!\mathrm{Vol}(Q) \\
&=(n+1)(\sum_{\ell=0}^nh_{\ell}) \\
&=(n+1)(1+h_1+h_2+\dots +h_n).
\end{align*}
On the other hand, $\mathrm{Card}(Q\cap \mathbb{Z}^n)=h_1+n+1$ by $(1)$ in Proposition $\ref{prop:EH}$.
Since all integers $h_{\ell} \; (\ell=1,\cdots,n)$ are nonnegative by Proposition $\ref{prop:EH}$, $(3)$, we conclude that
\[
\mathrm{Card}(Q\cap \mathbb{Z}^n)=h_1+n+1\leqslant (n+1)(1+h_1+h_2+\dots+h_n)=(n+1)!\mathrm{Vol}(Q).
\]
In particular, we see that $\mathrm{Card}(Q\cap \mathbb{Z}^n) = (n+1)! \mathrm{Vol}(Q)$ if and only if
\[
(h_0,\dots,h_n)=(1,0,\dots,0),                             
\]
and in this case we have that $(h_0,\dots,h_n)$ equals $(1,0,\dots,0)$ if and only if $Q=\Delta_n$.
The lemma is proved.
\end{proof}

\end{document}